\documentclass[12pt]{article}

\usepackage{amssymb}
\usepackage{amsmath}
\usepackage{amsthm}
\usepackage{mathrsfs}

\newtheorem{theorem}{Theorem}[section]
\newtheorem{lemma}[theorem]{Lemma}
\newtheorem{proposition}[theorem]{Proposition}

\theoremstyle{remark}
\newtheorem{remark}[theorem]{Remark}

\newcommand{\cadlag} {c\`{a}dl\`{a}g\ }
\newcommand{\Levy} {L\'{e}vy~}
\newcommand{\LevyIto}{L\'{e}vy-Ito~}
\newcommand{\ud}{\mathrm{d}}

\title{Exact conditions for no ruin for the generalised Ornstein-Uhlenbeck process
\thanks{The authors wish to dedicate this paper to the memory of Chris Heyde, mentor and friend.}}

\author{Damien Bankovsky \thanks{{Email: bankovd@maths.anu.edu.au Dept. of Mathematics, Australian National University.}}\and
  Allan Sly \thanks{{Email: sly@stat.berkeley.edu Dept. of Statistics, U.C. Berkeley. Supported by NSF grants DMS-0528488 and DMS-0548249 }}}

\begin{document}

\maketitle

\begin{abstract}
For a bivariate \Levy process $(\xi_t,\eta_t)_{t\geq 0}$ the
generalised Ornstein-Uhlenbeck (GOU) process is defined as
\[V_t:=e^{\xi_t}\left(z+\int_0^t e^{-\xi_{s-}}\ud
\eta_s\right),~~t\ge0,\]where $z\in\mathbb{R}.$ We define necessary
and sufficient conditions under which the infinite horizon ruin
probability for the process is zero. These conditions are stated in
terms of the canonical characteristics of the \Levy process and
reveal the effect of the dependence relationship between $\xi$ and
$\eta.$ We also present technical results which explain the
structure of the lower bound of the GOU.
\end{abstract}

\bigskip



\section{Introduction and Notation}

For a bivariate \Levy process $(\xi,\eta)=(\xi_t,\eta_t)_{t\geq 0}$
the generalised Ornstein-Uhlenbeck (GOU) process $V=(V_t)_{t\ge0},$
where $V_0=z\in\mathbb{R},$ is defined as
\begin{equation}\label{GOU definition} V_t:=e^{\xi_t}\left(z+\int_0^t e^{-\xi_{s-}}\ud \eta_s\right).\end{equation}
It is closely related to the stochastic integral process
$Z=(Z_t)_{t\ge0}$ defined as
\begin{equation}\label{Z definition} Z_t:=\int_0^t e^{-\xi_{s-}}\ud \eta_s.\end{equation}
The GOU is a time homogenous strong Markov process. For an overview
of its properties see Maller et al. \cite{MallerMullerSzimayer07},
and Carmona et al. \cite{CarmonaPetitYor01}. Applications are many,
and include option pricing (e.g Yor \cite{Yor01}), financial time
series (e.g. Kl\"{u}ppelberg et al.
\cite{KluppelbergLindnerMaller04}), insurance, and risk theory (e.g.
Paulsen \cite{Paulsen98}, Nyrhinen \cite{Nyrhinen01}).

In this paper, we present some basic foundational results on the
ruin probability for the GOU, in a very general setup. There are
only a few papers dealing with this, or with passage-time problems
for the GOU. Patie \cite{Patie05}, and Novikov \cite{Novikov03},
give first passage-time distributions in the special case that
$\xi_t=\lambda t$ for $\lambda\in\mathbb{R},$ and $\eta$ has no
positive jumps. With regard to ruin probability, Nyrhinen
\cite{Nyrhinen01} and Kalashnikov and Norberg
\cite{KalashnikovNorberg02} discretize the GOU into a stochastic
recurrence equation. Under a variety of conditions, they produce
some asymptotic equivalences for the infinite horizon ruin
probability. Other work on the GOU ruin probability comes from
Paulsen \cite{Paulsen98}. In the special case that $\xi$ and $\eta$
are independent, Paulsen gives conditions for certain ruin for the
GOU, and a formula for the ruin probability under conditions which
ensure that the integral process $Z_t$ converges almost surely as
$t\rightarrow\infty.$

Since these papers were written, the theory relating to the GOU, and
to the process $Z,$ has advanced.  In the general case where
dependence between $\xi$ and $\eta$ is allowed, Erickson and Maller
\cite{EricksonMaller05} present necessary and sufficient conditions
for the almost sure convergence of $Z_t$ to a random variable
$Z_\infty$ as $t\rightarrow\infty.$ Bertoin et al.
\cite{BertoinLindnerMaller07} present necessary and sufficient
conditions for continuity of the distribution of $Z_\infty$ given it
exists. Lindner and Maller \cite{LindnerMaller05} show that strict
stationarity of $V$ is equivalent to convergence of an integral
$\int_0^t e^{\xi_{s-}}\ud L_s,$ where $L$ is an auxiliary \Levy
process composed of elements of $\xi$ and $\eta.$ Note that in
\cite{LindnerMaller05} the sign of the process $\xi$ is reversed in
the definition of the GOU. For our purposes it suits to have the GOU
in the form $V_t:=e^{\xi_t}\left(z+Z_t\right)$ and to study the
behaviour of $V$ in terms of $Z.$

Our main results are presented in Section 2. Theorem \ref{ruin prob
corollary} presents exact necessary and sufficient conditions under
which the infinite horizon ruin probability for the GOU is zero.
These conditions do not relate to the convergence of $Z$ or
stationarity of $V$ or to any moment conditions. Instead they are
are expressed at a more basic level, directly on the \Levy measure
of $(\xi,\eta).$ Theorem \ref{ruin theorem} shows that $P(Z_t<0)>0$
for all $t>0$ as long as $\eta$ is not a subordinator. This result
is an important building block in the proof of Theorem \ref{ruin
prob corollary}. Finally in Section 2, Theorem \ref{ruin formula
theorem} extends a ruin probability formula in Paulsen
\cite{Paulsen98}, presenting a slightly different version which
deals with the general dependent case, and applies whenever $Z_t$
converges almost surely to a random variable $Z_\infty$ as
$t\rightarrow\infty.$

Section 3 contains technical results of interest, which
characterise what we call the lower bound function of the GOU, and
are used to prove the main ruin probability theorem. Section 4
contains proofs of the results stated in Sections 2 and 3.

\subsection{Notation}
We now set out our theoretical framework and notation. Let
$(X_t)_{t\geq 0}:=(\xi_t,\eta_t)_{t\geq 0}$ be a bivariate \Levy
process with $\xi_0=\eta_0=0,$ adapted to a filtered complete
probability space $(\Omega,
\mathscr{F},\mathbb{F}=(\mathscr{F}_t)_{0\le t\le\infty},P)$
satisfying the ``usual hypotheses'' (see Protter \cite{Protter04}
p.3), where $\xi$ and $\eta$ are not identically zero. Assume the
$\sigma$-algebra $\mathscr{F}$ and the filtration $\mathbb{F}$ are
generated by $(\xi,\eta)$, that is,
$\mathscr{F}:=\sigma\left((\xi,\eta)_t:0\le t<\infty\right)$ and
$\mathscr{F}_t:=\sigma\left((\xi,\eta)_s:0\le s\le t\right).$ Note
that the processes $V$ and $Z$ are defined with respect to
$\mathbb{F}.$

The characteristic triplet of $(\xi,\eta)$ will be written
$\left((\tilde{\gamma}_\xi,\tilde{\gamma}_\eta),\Sigma_{\xi,\eta},\Pi_{\xi,\eta}\right)$
where $(\tilde{\gamma}_\xi,\tilde{\gamma}_\eta)\in\mathbb{R}^2,$ the
Gaussian covariance matrix $\Sigma_{\xi,\eta}$ is a non-stochastic
$2\times2$ positive definite matrix, and the \Levy measure
$\Pi_{\xi,\eta}$ is a $\sigma-$finite measure on
$\mathbb{R}^2\setminus\{0\}$ satisfying the condition
$\int_{\mathbb{R}^2}\min\{|z|^2,1\}\Pi_{\xi,\eta}(\ud z)<\infty,$
where $|\cdot|$ denotes Euclidean distance. For details on \Levy
processes see Bertoin \cite{Bertoin96} and Sato \cite{Sato99}.

The \LevyIto decomposition (Sato \cite{Sato99}, Ch.4,) breaks down
$(\xi,\eta)$ into a sum of four mutually independent \Levy
processes:
\begin{eqnarray}\label{2 dim levy ito decomposition}\nonumber(\xi_t,\eta_t)&=&(\tilde{\gamma}_\xi,\tilde{\gamma}_\eta)t+
(B_{\xi,t},B_{\eta,t})+\int_{|z|<1}z\left(N_{\xi,\eta,t}(\cdot,\ud
z)-t\Pi_{\xi,\eta}(\ud
z)\right)\\
&&~~+\int_{|z|\ge1}zN_{\xi,\eta,t}(\cdot,\ud z),
\end{eqnarray}
where $B_\xi$ and $B_\eta$ are Brownian motions such that
$(B_\xi,B_\eta)$ has covariance matrix $\Sigma_{\xi,\eta},$ and
$N_{\xi,\eta,t}(\omega,~)$ is the random jump measure of
$(\xi,\eta)$ such that
$E\left(N_{\xi,\eta,1}(\omega,\Lambda)\right)=\Pi_{\xi,\eta}(\Lambda)$
for $\Lambda$ a Borel subset of $\mathbb{R}^2\setminus\{0\}$ whose
closure does not contain $0.$ We can write (see Protter
\cite{Protter04}, p.31)
\begin{equation}(\tilde{\gamma}_\xi,\tilde{\gamma}_\eta)=E\left((\xi_1,\eta_1)-\int_{|z|\ge1}zN_{\xi,\eta,1}(\cdot,\ud z)\right).
\end{equation}
The characteristic triplets of $\xi$ and $\eta$ as one-dimensional
\Levy processes are denoted $(\gamma_\xi,\sigma_\xi^2,\Pi_\xi)$ and
$(\gamma_\eta,\sigma_\eta^2,\Pi_\eta)$ respectively, where
\begin{equation}\Pi_\xi(\Gamma)=\Pi_{\xi,\eta}(\Gamma\times\mathbb{R})\hbox{~~and~~}
\Pi_\xi(\Gamma)=\Pi_{\xi,\eta}(\mathbb{R}\times\Gamma)
\end{equation} for $\Gamma$ a Borel subset of $\mathbb{R}\setminus\{0\}$ whose
closure does not contain $0,$
\begin{equation}\label{2 dim to 1 dim equation}(\gamma_\xi,\gamma_\eta)=(\tilde{\gamma}_\xi,\tilde{\gamma}_\eta)+
\int_{\{|x|\le 1,|y|>\sqrt{1-x^2}\}}(x,y)\Pi_{\xi,\eta}(\ud(x,y)),
\end{equation}
and $\sigma_\xi^2$ and $\sigma_\eta^2$ are the upper left and lower
right entries respectively, in the matrix $\Sigma_{\xi,\eta}.$
Analagous to (\ref{2 dim levy ito decomposition}), we can write the
\LevyIto decomposition of $\xi$ as
\begin{equation}\xi_t=\gamma_\xi t+B_{\xi,t}+\int_{|x|<1}x\left(N_{\xi,t}(\cdot,\ud x)-t\Pi_\xi(\ud x)\right)+\int_{|x|\ge 1}xN_{\xi,t}(\cdot,\ud x),
\end{equation}
where
\begin{equation}\label{1 dim drift equation}\gamma_\xi=E\left(\xi_1-\int_{|x|\ge 1}xN_{\xi,1}(\cdot,\ud
x)\right),
\end{equation}
and similarly for $\eta.$ For further details on \LevyIto
decompositions, see Sato \cite{Sato99}, Chapter 4.

A \Levy process is said to be a subordinator if it takes only
non-negative values, which implies that its sample paths are
non-decreasing (Bertoin \cite{Bertoin96}, p.71).

Stochastic integrals are interpreted according to Protter
\cite{Protter04}. The integral $\int_a^b$ is interpreted as
$\int_{[a,b]}$ and the integral $\int_{a+}^b$ as $\int_{(a,b]}.$ The
jump of a process $Y$ at $t$ is denoted by $\Delta Y_t:=Y_t-Y_{t-}.$
The \Levy measure of a \Levy process $Y$ is denoted by $\Pi_Y.$ If
$T$ is a fixed time or a stopping time denote the process $Y$
stopped at $T$ by $Y^T$ and define it by $Y_t^T:=Y_{t\wedge
T}:=Y_{\min\{t,T\}}.$ For a function $f(x)$ define
$f^+(x):=f(x)\vee0:=\max\{f(x),0\}$ and $f^-(x):=\max\{-f(x),0\}.$
The symbol $1_\Lambda$ will denote the characteristic function of a
set $\Lambda.$ The symbol $=_D$ will denote equality in distribution
of two random variables. The initials ``iff'' will denote the phrase
``if and only if''. The symbol ``a.s'' will denote equality, or
convergence, almost surely. Let $T_z$ denote the first time $V$
drops below zero, so
$$T_z:=\inf\left\{t>0:V_t<0\big|V_0=z\right\}$$ and $T_z:=\infty$ whenever
$V_t>0 ~~\forall t>0$ and $V_0=z.$ For $z\ge 0,$ define the infinite
horizon ruin probability function to be
$$\psi(z):=P\left(\inf_{t\geq 0}
V_t<0\big|V_0=z\right)=P(T_z<\infty).$$

\section{Ruin Probability Results}
Our results are given in terms of regions of support of the \Levy
measure $\Pi_{\xi,\eta}$. We define some notation, beginning with
the following quadrants of the plane. Let
$A_1:=\left\{(x,y)\in\mathbb{R}^2:x\ge 0,y\ge0\right\},$ and
similarly, let $A_2,$ $A_3$ and $A_4$ be the quadrants in which
$\{x\ge0,y\le0\},$ $\{x\le0,y\le0\}$ and $\{x\le0,y\ge0\}$
respectively. For each $i=1,2,3,4$ and $u\in\mathbb{R}$ define
$$A_i^u:=\left\{(x,y)\in A_i:y-u(e^{-x}-1)<0\right\}.$$

These sets are defined such that if $(\Delta\xi_t,\Delta\eta_t)\in
A_i^u$ and $V_{t-}=u,$ then $\Delta V_t<0,$ as we see from the
equation
\begin{eqnarray}\label{jump equation 1}\nonumber\Delta V_t&=&V_t-V_{t-}\\&=&
e^{\xi_t}\Big(z+\int_0^{t-} e^{-\xi_{s-}}\ud \eta_s+e^{-\xi_t}\Delta\eta_t\Big)-e^{\xi_{t-}}\Big(z+\int_0^{t-} e^{-\xi_{s-}}\ud \eta_s\Big)\nonumber\\
&=&(e^{\xi_t}-e^{\xi_{t-}})\Big(z+\int_0^{t-} e^{-\xi_{s-}}\ud \eta_s\Big)+e^{\xi_t}e^{\xi_{t-}}\Delta\eta_t\nonumber\\
&=&(e^{\Delta\xi_t}-1)V_{t-}+e^{\Delta\xi_t}\Delta\eta_t.
\end{eqnarray}
If  $u\le 0$ then $A_2^u=A_2$ and $A_4^u=\emptyset.$ As $u$
decreases to $-\infty$, the sets $A_1^u$ shrink, whilst $A_3^u$
expand. Define
\[
\theta_1:= \left\{ \begin{array}{ll}
\sup\left\{u\le 0:\Pi_{\xi,\eta}(A_1^u)>0\right\} & \\
-\infty~~~~\textrm{if~}\Pi_{\xi,\eta}(A_1)=0 &
\end{array} \right.\,,
\theta_3:= \left\{ \begin{array}{ll}
\inf\left\{u\le 0:\Pi_{\xi,\eta}(A_3^u)>0\right\} &\\
0~~~~\textrm{if~}\Pi_{\xi,\eta}(A_3)=0 &
\end{array} \right.\,.
\]
If $u\ge 0$ then $A_3^u=A_3$ and $A_1^u=\emptyset.$ As $u$ increases
to $\infty$, the sets $A_2^u$ shrink, whilst $A_4^u$ expand. Define
\[
\theta_2:= \left\{ \begin{array}{ll}
\sup\left\{u\ge 0:\Pi_{\xi,\eta}(A_2^u)>0\right\} & \\
0~~~~\textrm{if~}\Pi_{\xi,\eta}(A_2)=0, &
\end{array} \right.\,,
\theta_4:= \left\{ \begin{array}{ll}
\inf\left\{u\ge 0:\Pi_{\xi,\eta}(A_4^u)>0\right\} &\\
\infty~~~~\textrm{if~}\Pi_{\xi,\eta}(A_4)=0 &
\end{array} \right.\,.
\]
For each $i=1,2,3,4,$ note that $\Pi_{\xi,\eta}(A_i^{\theta_i})=0,$
since in the definitions of $A_i^u$ we are requiring that
$y-u(e^{-x}-1)$ be strictly less than zero.

\begin{theorem}[Exact conditions for no ruin for the GOU]\label{ruin prob corollary}
The ruin probability $\psi(z)=0$ for large enough $z\ge0$ if and
only if the \Levy measure satisfies $\Pi_{\xi,\eta}(A_3)=0,$
$\theta_2\le\theta_4,$ and:
\begin{itemize}
\item when $\sigma_\xi^2\neq0$ the Gaussian covariance matrix is of form $\Sigma_{\xi,\eta}= \left[
  \begin{array}{ r r }
     1 & -u \\
     -u & u^2
  \end{array} \right]\sigma_\xi^2$ for some
  $u\in[\theta_2,\theta_4]$ satisfying
\begin{equation}\label{finite drift equation}\tilde{\gamma}_\eta+u\tilde{\gamma_\xi}-\frac{1}{2}u\sigma_\xi^2-
\int_{\{y-u(e^{-x}-1)>0\}\cap\{x^2+y^2<1\}}(ux+y)\Pi_{\xi,\eta}(\ud(x,y))\ge0;
\end{equation}
\item when $\sigma_\xi^2=0$ the Gaussian covariance matrix is of form
$\Sigma_{\xi,\eta}=0$ and there exists $u\in[\theta_2,\theta_4]$
satisfying (\ref{finite drift equation}).
\end{itemize}
If $\sigma_\xi^2\neq0$ and the conditions of the theorem hold, then
$\psi(z)=0$ for all $z\ge u:=\frac{\sigma_\eta}{\sigma_\xi},$ whilst
$\psi(z)>0$ for all $z<u.$

If $\sigma_\xi^2=0$ and the conditions of the theorem hold, then
$\psi(z)=0$ for all $z\ge
u':=\max\left\{\theta_2,\inf\{u>0:(\ref{finite drift
equation})\hbox{~holds}\}\right\},$ whilst $\psi(z)>0$ for all
$z<u'.$
\end{theorem}
We now discuss some examples and special cases which illustrate and
amplify the results in Theorem \ref{ruin prob corollary}.
\begin{remark}\label{main remarks}
\begin{enumerate}

{\rm\item Suppose that $(\xi,\eta)$ is continuous. We can then write
$(\xi_t,\eta_t)=(\gamma_\xi t,\gamma_\eta
t)+(B_{\xi,t},B_{\eta,t}).$ Theorem \ref{ruin prob corollary} states
that $\psi(z)=0$ for all $z\ge u$ and $\psi(z)>0$ for all $z<u,$ if
and only if there exists $u>0$ such that $B_\eta=-uB_\xi,$ and
$(\gamma_\xi-\frac{1}{2}\sigma_\xi^2)u+\gamma_\eta\ge0.$ For example
we could have
\begin{equation}\label{continuous
example}(\xi_t,\eta_t):=(B_t+ct,-B_t+(1/2-c)t),
\end{equation} where $c\in\mathbb{R}.$ Then Theorem \ref{ruin prob corollary} implies that $\psi(z)=0$ for
all $z\ge u=\frac{\sigma_\eta}{\sigma_\xi}=1$ whilst $\psi(z)>0$ for
all $z<1.$ In this simple case, we can check the result directly.
Using Ito's formula we obtain
$$Z_t=-\int_0^te^{-(B_s+cs)}\ud B_s+(1/2-c)\int_0^te^{-(B_s+cs)}\ud s=e^{-(B_s+cs)}-1,$$ and hence a lower bound for $Z$ is
$-1.$

\item Suppose that $(\xi,\eta)$ is a finite variation \Levy process. Then we must have $\Sigma_{\xi,\eta}=0$ and
$\int_{|z|<1}|z|\Pi_{\xi,\eta}(\ud z)<\infty$.
We can define the drift vector as
\[(d_\xi,d_\eta):=\gamma_\eta-\int_{|z|<1}z\Pi_{\xi,\eta}(\ud z)
\] and write
\[(\xi_t,\eta_t)=(d_\xi,d_\eta)t+\int_{\mathbb{R}^2}zN_{\xi,\eta,t}(\cdot,\ud z).
\]
In this situation, the conditions of Theorem \ref{ruin prob
corollary} can be made more explicit. Theorem \ref{ruin prob
corollary} states that $\psi(z)=0$ for large enough $z$ if and only
if $\Pi_{\xi,\eta}(A_3)=0,$ $\theta_2\le\theta_4,$ and at least one
of the following is true:
\begin{itemize}
\item $d_\xi=0,$ and $d_\eta\ge 0$; or
\item $d_\xi>0$ and $-\frac{d_\eta}{d_\xi}\le\theta_4$; or
\item $d_\eta>0,$ and $d_\xi<0,$ such that $-\frac{d_\eta}{d_\xi}\ge\theta_2.$
\end{itemize}
If the second property holds, then $\psi(z)=0$ for all
$z\ge\max\{\theta_2,-\frac{d_\eta}{d_\xi}\}$ and $\psi(z)>0$ for all
$z<\max\{\theta_2,-\frac{d_\eta}{d_\xi}\}.$ If the other properties
hold, then $\psi(z)=0$ for all $z\ge\theta_2$ and $\psi(z)>0$ for
all $z<\theta_2.$

These results follow easily by transforming condition (\ref{finite
drift equation}) into conditions on $(d_\xi,d_\eta).$ For a simple
example, let $N_t$ be a Poisson process with parameter $\lambda,$
let $c>0$ and let
\begin{equation}\label{jump example}(\xi_t,\eta_t):=(-ct+N_t,2ct-N_t).
\end{equation}
Then we are in the third case above, and $\psi(z)=0$ for all
$z\ge\theta_2=\frac{e}{e-1},$ and $\psi(z)>0$ for all
$z<\frac{e}{e-1}.$ In this simple case, we can verify the results by
direct but tedious calculations which we omit here.

\item The case in which $\xi$ and $\eta$ are independent is analysed in Paulsen
\cite{Paulsen98}. In the cases $E(\xi_1)<0$ and $E(\xi_1)=0,$ and
under certain moment conditions, he shows that $\psi(z)=1$ for all
$z\ge 0.$ Theorem \ref{ruin prob corollary} shows that the situation
changes when dependence is allowed. The continuous process defined
in (\ref{continuous example}), and the jump process defined in
(\ref{jump example}), illustrate this difference. Each process
trivially satisfies Paulsen's moment conditions and can satisfy
$E(\xi_1)<0,$ or $E(\xi_1)=0,$ depending on the choices of $c$ and
$\lambda,$ however it is not the case that $\psi(z)=1$ for all $z\ge
0.$

\item If $\eta$ is a subordinator then $Z_t\ge 0$ for all $t\ge0,$
and hence $\psi(z)=0$ for all $z\ge 0.$ Theorem \ref{ruin prob
corollary} agrees with this trivial case. By Sato \cite{Sato99},
p.137, $\eta$ is a subordinator if and only if the following three
conditions hold:
\begin{itemize}
\item $\sigma_\eta^2=0$, so $\eta$
has no Brownian component;
\item $\Pi_\eta((-\infty,0))=0,$ so
$\eta$ has no negative jumps;
\item $d_\eta\ge 0,$ where
\[
d_\eta:=\gamma_\eta-\int_{(0,1)}y\Pi_\eta(\ud
y)=E\left(\eta_1-\int_{(0,\infty)}yN_{\eta,1}(\cdot,\ud y)\right).\]
Note that when $\Pi_\eta((-\infty,0))=0, $then $d_\eta$ exists and
$d_\eta\in[-\infty,\infty),$ where $d_\eta=-\infty$ iff
$\int_{(0,1)}y\Pi_\eta(\ud y)=\infty$.
\end{itemize}
Now $\sigma_\eta^2=0$ implies that $\Sigma_{\xi,\eta}= \left[
  \begin{array}{ r r }
     1 & 0 \\
     0 & 0
  \end{array} \right]\sigma_\xi^2.$ When $\eta$ has no negative
  jumps, then
$\Pi_{\xi,\eta}(A_3)=0=\Pi_{\xi,\eta}(A_2),$ and hence
$0=\theta_2\le\theta_4.$ The third property, $d_\eta\ge0,$ implies
that (\ref{finite drift equation}) is satisfied for $u=0,$ since
(\ref{2 dim to 1 dim equation}) implies that
\begin{eqnarray*}\tilde{\gamma}_\eta-\int_{\{y>0\}\cap\{x^2+y^2<1\}}y\Pi_{\xi,\eta}(\ud(x,y))
&=&\gamma_\eta-\int_{(-1,1)\times(0,1)}y\Pi_{\xi,\eta}\left(\ud(x,y)\right)\\
&=&d_\eta.
\end{eqnarray*}
Hence, Theorem \ref{ruin prob corollary} verifies that $\psi(z)=0$
for all $z\ge u=0.$

\item The expression on the left hand side of (\ref{finite drift equation})
always exists whenever the remaining conditions of the theorem are
satisfied, however it may have the value $-\infty.$ If all
conditions of the theorem are satisfied then
\begin{equation}\label{finite variation
equation}\int_{\{y-u(e^{-x}-1)\in(0,1)\}}\left(y-u(e^{-x}-1)\right)\Pi_{\xi,\eta}(\ud(x,y))<\infty.
\end{equation}
On first viewing, (\ref{finite variation equation}) may seem
counterintuitive, as it places a constraint on the size of the
positive jumps of $V.$ However, if (\ref{finite variation equation})
does not hold, and all the other conditions, excluding (\ref{finite
drift equation}), are satisfied, then the \Levy properties of
$(\xi,\eta)$ imply that $V_t$ can drift negatively when $V_{t-}=u.$
These statements are discussed further in Remark \ref{infinite
variation remark} following Theorem \ref{sub conditions theorem}. }
\end{enumerate}
\end{remark}

\begin{theorem}\label{ruin theorem} The \Levy process $\eta$
is not a subordinator if and only if $P(Z_T<0)>0$ for any fixed time
$T>0.$
\end{theorem}

One direction of this result is trivial and has been noted above,
namely, if $\eta$ is a subordinator then $P(Z_T<0)=0$ for any $T>0.$
The other direction seems quite intuitive and in fact is implicitly
assumed by Paulsen \cite{Paulsen98} in the case when $\xi$ and
$\eta$ are independent. However even in the independent case the
proof is non-trivial. We prove it in the general case using a change
of measure argument and some analytic lemmas. As well as being of
independent interest, this result is essential in proving Theorem
\ref{ruin prob corollary}.

The final theorem in this section provides a formula for the ruin
probability in the case that $Z$ converges. Recall that $T_z$
denotes the first time $V$ drops below zero when $V_0=z,$ or
equivalently, the first time $Z$ drops below $-z.$

\begin{theorem}\label{ruin formula theorem}Suppose $Z_t$ converges a.s to a finite random variable $Z_\infty$ as $t\rightarrow\infty,$
and let $G(z):=P(Z_\infty\le z).$ Then
$$\psi(z)=\frac{G(-z)}{E\left(G(-V_{T_z})\big|T_z<\infty\right)}.$$
Note that
$G(-V_{T_z})(\omega):=P\left(\nu\in\Omega~:~Z_\infty(\nu)<-V_{T_z}(\omega)
\right).$ It is defined whenever $T_z(\omega)<\infty.$
\end{theorem}

\begin{remark}\label{formula remarks}
\begin{enumerate}
{\rm\item In the case that $\xi$ and $\eta$ are independent, Paulsen
\cite{Paulsen98} shows, under a number of side conditions which
ensure that $Z_t$ converges a.s to a finite random variable
$Z_\infty$  with distribution function $H(z):=P(Z_\infty<z)$ as
$t\rightarrow\infty$, that
$$\psi(z)=\frac{H(-z)}{E\left(H(-V_{T_z})\big|T_z<\infty\right)}.$$
This formula is a modification of a result given by Harrison
\cite{Harrison77} for the special case in which $\xi$ is
deterministic drift and $\eta$ is a \Levy process with finite
variance. Theorem \ref{ruin formula theorem} extends the formula to
the general dependent case. Our proof is similar to those of Paulsen
and Harrison, however we write it out in full because some details
are different.

\item Erickson and Maller \cite{EricksonMaller05} prove that $Z_t$ converges a.s to
a finite random variable $Z_\infty$ as $t\rightarrow\infty$ if and
only if
$$\lim_{t\rightarrow\infty}\xi_t=+\infty~a.s~~\hbox{and}~~
\int_{\mathbb{R}\setminus[-e,e]}\left(\frac{\ln |y|}{A_\xi(\ln
|y|)}\right)\Pi_\eta(\ud y)<\infty,$$ where, for $x\ge1,$
$$A_\xi(x):=1+\int_1^x\Pi_\xi((z,\infty))\ud z.$$
Lindner and Maller \cite{LindnerMaller05} prove that if $V$ is not a
constant process, then $V$ is strictly stationary if and only if
$\int_0^\infty e^{\xi_{s-}}\ud L_s$ converges a.s to a finite random
variable as $t\rightarrow\infty$, where $L$ is the \Levy process
$$L_t:=\eta_t+\sum_{0<s\le t}\left( e^{-\Delta\xi_s}-1\right)\Delta\eta_s-t\hbox{Cov}(B_{\xi,1},B_{\eta,1}),~~~t\ge0.$$

In neither of these cases do the conditions of Theorem \ref{ruin
prob corollary} simplify. Each of the processes defined in
(\ref{continuous example}) and (\ref{jump example}) can belong to
either of these cases, or neither, depending on the choice of
constant $c$ and parameter $\lambda.$

\item Bertoin et al. \cite{BertoinLindnerMaller07} prove that if $Z_t$ converges a.s to
a finite random variable $Z_\infty$ as $t\rightarrow\infty$, then
$Z_\infty$ has an atom iff $Z_\infty$ is a constant value $k$ iff
$P\left(Z_t=k(1-e^{-\xi_t})~\forall t>0\right)=1$ iff
$e^{-\xi}=\epsilon(-\eta/k),$ where $\epsilon(\cdot)$ denotes the
stochastic exponential. In this case it is trivial that $\psi(z)=0$
for all $z\ge-k.$ Theorem \ref{ruin prob corollary} produces the
same result, however this will not become immediately clear until
Remark \ref{lower bound remarks} (2) following Theorem \ref{lower
bound theorem}. }
\end{enumerate}
\end{remark}

\section{Technical Results of Interest}

This section contains technical results needed in the proofs of
Theorems \ref{ruin prob corollary} and \ref{ruin theorem}, which
also have some independent interest. Recall that the stochastic, or
Dol\'{e}ans-Dade, exponential of a semimartingale $W_t$ is denoted
by $\epsilon(W)_t.$

\begin{proposition}\label{W conditions lemma}
Given a \Levy process $\xi$ with characteristic triplet
$(\gamma_\xi,\sigma_\xi,\Pi_\xi)$ there exists a \Levy process $W$
adapted to the same filtration, such that
$e^{-\xi_t}=\epsilon(W)_t$, where $(\xi,W)$ is the bivariate \Levy
process with characteristic triplet
$\left((\tilde{\gamma}_\xi,\tilde{\gamma}_W),\Sigma_{\xi,W},\Pi_{\xi,W}\right)$
defined as follows:
\begin{equation}\label{second covariance matrix equation}\Sigma_{\xi,W}= \left[
  \begin{array}{ r r }
     1 & -1 \\
     -1 & 1
  \end{array} \right]\sigma_\xi^2,
\end{equation}
the \Levy measure $\Pi_{\xi,W}$ is concentrated on
$\{(x,e^{-x}-1)~:~x\in\mathbb{R}\}$ so that
$$\Pi_W((-\infty,-1])=0$$ and
$$\Pi_W(\Lambda)=\Pi_\xi(-\ln(\Lambda+1))~~~~\textrm{when~}
\Lambda\subset(-1,\infty),$$ and
\begin{equation}\label{two dim equation}\tilde{\gamma}_\xi+\tilde{\gamma}_W=\frac{1}{2}\sigma_\xi^2+\int_{x^2+(e^{-x}-1)^2<1}(x+e^{-x}-1)\Pi_\xi(\ud x).\end{equation}
\end{proposition}

We define the lower bound function $\delta$ for $V$ in (\ref{GOU
definition}) as
\[
\delta(z) = \inf \left\{u\in\mathbb{R}:P\left(\inf_{t\geq 0} V_t
\leq u \big|V_0=z\right) > 0 \right\}.
\] The following theorem exactly characterizes the lower bound function.

\begin{theorem}\label{lower bound theorem}
The lower bound function satisfies the following properties:
\begin{enumerate}
\item \label{p:decrease} For all $z\in\mathbb{R}$, $\delta(z)\leq z.$
\item \label{p:monotone}If $z_1<z_2$ then $\delta(z_1) \leq \delta(z_2).$
\item \label{p:subordinator} Let $W$ be the \Levy process such that $e^{-\xi_t}=\epsilon(W)_t$.
Then $\delta(z)=z$ if and only if $\eta-z W$ is a subordinator.
\item \label{p:fixed} For all $z\in\mathbb{R}$, $\delta(z)=\delta(\delta(z)),$ and
\[
\delta(z)=\sup \left\{u : u \leq z, \eta-u W \hbox{~is a
subordinator}\right\}.
\]
\end{enumerate}
\end{theorem}

\begin{remark}\label{lower bound remarks}
\begin{enumerate}
{\rm\item If $\eta$ is a subordinator then $\delta(0)=0,$ so $V$
cannot drop below zero when $V_0=z\ge0.$
\item As noted in Remark \ref{formula remarks} (3), if $Z_t$ converges a.s to
a finite random variable $Z_\infty$ as $t\rightarrow\infty$, then
$Z_\infty$ has an atom iff $e^{-\xi}=\epsilon(-\eta/k).$ If this
holds then  $\delta(-k)=-k,$ since $\eta+k\left(-\eta/k\right)=0$
and hence is a subordinator. Thus $\psi(z)=0$ for all $z\ge-k,$ as
mentioned in Remark \ref{formula remarks} (3). }
\end{enumerate}
\end{remark}

\begin{theorem}\label{sub conditions theorem} Let $u\in\mathbb{R}.$ With $W$ defined as in Proposition \ref{W conditions lemma}, the \Levy
process $\eta-uW$ is a subordinator if and only if the following
three conditions are satisfied: the Gaussian covariance matrix is of the
form \begin{equation}\label{covariance matrix
equation}\Sigma_{\xi,\eta}= \left[
  \begin{array}{ r r }
     1 & -u \\
     -u & u^2
  \end{array} \right]\sigma_\xi^2,\end{equation} at least one of the following is true:
\begin{itemize}
\item $\Pi_{\xi,\eta}(A_3)=0$ and $\theta_2\le\theta_4$ and $u\in[\theta_2,\theta_4];$
\item $\Pi_{\xi,\eta}(A_2)=0$ and $\theta_1\le\theta_3$ and $u\in[\theta_1,\theta_3];$
\item $\Pi_{\xi,\eta}(A_3)=\Pi_{\xi,\eta}(A_2)=0$ and $u\in[\theta_1,\theta_4];$
\end{itemize} and in addition, $u$ satisfies (\ref{finite drift equation}).
\end{theorem}

\begin{remark}\label{infinite variation remark} {\rm In Remark \ref{main remarks} (4)
we stated three necessary and sufficient conditions for a \Levy
process to be a subordinator. These three conditions correspond
respectively with the three conditions in Theorem \ref{sub
conditions theorem}, as we shall see in the proof. In particular,
(\ref{finite drift equation}) is equivalent to the condition
$d_{\eta-uW}\ge0$. As noted in Remark \ref{main remarks} (4), if the
first two conditions of Theorem \ref{sub conditions theorem} hold,
then $d_{\eta-uW}\in[-\infty,\infty),$ thus ensuring that
(\ref{finite drift equation}) is well defined. Further, if all three
conditions hold, then $\int_{(0,1)}z\Pi_{\eta-uW}(\ud z)<\infty,$
which we will show to be equivalent to (\ref{finite variation
equation}). Note that if $\eta-uW$ has no Brownian component, no
negative jumps, but $\int_{(0,1)}z\Pi_{\eta-uW}(\ud z)=\infty,$
then, somewhat suprisingly, $\eta-uW$ is fluctuating and hence not a
subordinator, regardless of the value of the shift constant
$\gamma_{\eta-uW}.$ This behaviour occurs since
$d_{\eta-uW}=-\infty,$ and is explained in Sato \cite{Sato99}, p138.
}
\end{remark}

\section{Proofs}

We begin by proving Theorem \ref{ruin theorem}. For this proof, some
lemmas are required. In these we assume that $X=(\xi,\eta)$
has bounded jumps so that $X$ has finite absolute moments of all
orders. Then, to prove Theorem \ref{ruin theorem} we reduce to this case.

\begin{lemma}\label{martingale lemma} Suppose $X=(\xi,\eta)$ has bounded jumps and $E(\eta_1)=0.$ If we let $T>0$ be a fixed time then $Z^T$ is a mean-zero martingale with respect to $\mathbb{F}.$
\end{lemma}
\begin{proof} Since $\eta$ is a \Levy process the assumption $E(\eta_1)=0$ implies that $\eta$ is a \cadlag martingale. Since $\xi$ is \cadlag, $e^{-\xi}$ is
a locally bounded process and hence $Z$ is a local martingale for
$\mathbb{F}$ by Protter \cite{Protter04}, p.171. If we show that
$E\left(\sup_{s\le t}|Z_s^T|\right)<\infty$ for every $t\ge 0$ then
Protter \cite{Protter04}, p.38 implies that $Z^T$ is a martingale.
This is equivalent to showing $E\left(\sup_{t\le
T}|Z_t|\right)<\infty.$ Since $Z$ is a local martingale and $Z_0=0,$
the Burkholder-Davis-Gundy inequalities in Lipster and Shiryaev
\cite{LipsterShiryayev89}, p.70 and p.75, ensure the existence of
$b>0$ such that
\begin{eqnarray*}E\left(\sup_{0\le t\le T}\left|\int_0^t
e^{-\xi_{s-}}\ud\eta_s\right|\right)
&\le& bE\left(\left[\int_0^\bullet e^{-\xi_{s-}}\ud\eta_s,\int_0^\bullet e^{-\xi_{s-}}\ud\eta_s\right]_T^{1/2}\right)\\
&=& bE\left(\left(\int_0^T e^{-2\xi_{s-}}\ud
[\eta,\eta]_s\right)^{1/2}\right)\\
&\le&bE\left(\left(\int_0^T \sup_{0\le t\le T}e^{-2\xi_t}\ud
[\eta,\eta]_s\right)^{1/2} \right)\\
&=&bE\left(\sup_{0\le t\le T}e^{-\xi_{t}}[\eta,\eta]_T^{1/2}\right)\\
&\le&b\left(E\left(\sup_{0\le t\le
T}e^{-2\xi_{t}}\right)\right)^{1/2}\left(E\left([\eta,\eta]_T\right)\right)^{1/2},
\end{eqnarray*}
where the second inequality follows from the fact that
$[\eta,\eta]_s$ is increasing and the final inequality follows by
the Cauchy-Schwarz inequality. (The notation $[\cdot,\cdot]$ denotes the quadratic variation process.) Now
\[
E\left([\eta,\eta]_T\right) = \sigma_\eta^2 T + E\left(\sum_{0\leq s
\leq T} (\Delta \eta)^2\right) = \sigma_\eta^2 T + T \int x^2
\Pi_\eta(\ud x),
\]
which is finite since $\eta$ has bounded jumps. Thus it suffices to
prove $E\left(\sup_{0\le t\le T}e^{-2\xi_{t}}\right)<\infty$.
Setting $Y_t =e^{-\xi_t}/E(e^{-\xi_t})$, a non-negative martingale,
it follows by Doob's maximal inequality, as expressed in Shiryaev
\cite{EncyclopaediaofmathematicsV690}, p.765, that
\begin{eqnarray*}E\left(\sup_{0\le t\le
T}\frac{e^{-2\xi_t}}{\left(E(e^{-\xi_t})\right)^2} \right) &\le&4\frac{E\left(e^{-2\xi_T}\right)}{\left(E(e^{-\xi_T}) \right)^2},\\
\end{eqnarray*}
which is finite since $\xi$ has bounded jumps and hence has finite
exponential moments of all orders (Sato \cite{Sato99}, p.161). It is
shown in Sato \cite{Sato99}, p.165, that
$\left(E(e^{-\xi_t})\right)^2=\left(E(e^{-\xi_1}) \right)^{2t}.$
Letting $c:=\left(E(e^{-\xi_1}) \right)^2\in(0,\infty),$ the above
inequality implies that
$$E\left(\sup_{0\le t\le T}e^{-2\xi_{t}}\right)\le\max\{1,c^T\}E\left(\sup_{0\le t\le T}\frac{e^{-2\xi_t}}{c^t}
\right)<\infty.$$\end{proof}

We now present two lemmas dealing with absolute continuity of
measures. These lemmas will be used to construct a new process $W$
such that $W^T$ is a mean-zero martingale which is mutually
absolutely continuous with $Z^T.$ Then  $P(Z_T<0)>0$ if and only if
$P(W_T<0)>0,$ and the latter statement will follow immediately from
the fact that $W^T$ is a mean-zero martingale.

\begin{lemma}\label{measure equivalence lemma}
Let $X:=(\xi,\eta)$ and $Y:=(\tau,\nu)$ be bivariate \Levy processes
adapted to $(\Omega,\mathscr{F},\mathbb{F},P)$, and let
$Z_t:=\int_0^t e^{-\xi_{s-}}\ud \eta_s$ and $W_t:=\int_0^t
e^{-\tau_{s-}}\ud \nu_s.$ If the induced probability measures of
$X^T$ and $Y^T$ are mutually absolutely continuous, then the induced
probability measures of $Z^T$ and $W^T$ are mutually absolutely
continuous.
\end{lemma}

\begin{proof}
Let $D([0,T]\rightarrow\mathbb{R}^2)$ denote the set of \cadlag
functions from $[0,T]$ to $\mathbb{R}^2$ and $\mathscr{B}^{2[0,T]}$
denote the $\sigma$-algebra generated in this set by the Borel
cylinder sets (see Kallenberg \cite{Kallenberg97}). Then the induced
probability measures of $X^T$ and $Y^T$ can be written as $P_{X^T}$
and $P_{Y^T}$ on the measure space
$\left(D([0,T]\rightarrow\mathbb{R}^2),\mathscr{B}^{2[0,T]}\right).$
Let $C:=(C',C'')$ be the co-ordinate mapping of
$\left(D([0,T]\rightarrow\mathbb{R}^2),\mathscr{B}^{2[0,T]}\right)$
to itself. Define the process $Z'$ on the probability space
$\left(D([0,T]\rightarrow\mathbb{R}^2),\mathscr{B}^{2[0,T]},P_{X^T}\right)$
by $Z'_t:=\int_0^t e^{-C'_{s-}}\ud C''_s.$ Define $W'$ on
$\left(D([0,T]\rightarrow\mathbb{R}^2),\mathscr{B}^{2[0,T]},P_{Y^T}\right)$
by  $W'_t:=\int_0^t e^{-C'_{s-}}\ud C''_s.$ Note that $Z'$ and $W'$
are different processes since they are being evaluated under
different measures. Now $Z=X\circ Z'$ and $W=Y\circ W'$. Hence
$P(Z^T\in \Lambda)=P_{X^T}(Z'\in \Lambda)$ and $P(W^T\in
\Lambda)=P_{Y^T}(W'\in \Lambda).$ Since $P_{X^T}$ and $P_{Y^T}$ are
mutually absolutely continuous, Protter \cite{Protter04}, p.60
implies that $Z'$ and $W'$ are $P_{X^T}$-indistinguishable, and
$P_{Y^T}$-indistinguishable. So $P_{X^T}(Z'\in
\Lambda)=P_{X^T}(W'\in \Lambda).$ Since $P_{X^T}$ and $P_{Y^T}$ are
mutually absolutely continuous $P_{X^T}(W'\in \Lambda)=0$ iff
$P_{Y^T}(W'\in \Lambda)=0$ which proves $P(Z^T\in\Lambda)=0$ iff
$P(W^T\in\Lambda)=0,$ as required.\end{proof}

\begin{lemma}\label{existence of Y lemma}
If $X:=(\xi,\eta)$ has bounded jumps, $E(\eta_1)\ge 0,$ $\eta$ is
not a subordinator, and $\eta$ is not pure deterministic drift, then
there exists a bivariate \Levy process $Y:=(\tau,\nu)$ with bounded
jumps, adapted to $(\Omega,\mathscr{F},\mathbb{F},P),$ such that
$X^T$ and $Y^T$ are mutually absolutely continuous for all $T>0$,
and $E(\nu_1)=0.$
\end{lemma}

\begin{proof}
As mentioned in Remark \ref{main remarks} (4,) the \Levy process $\eta$ is a subordinator if and only if the
following three conditions hold: $\sigma_\eta^2=0$,
$\Pi_\eta((-\infty,0))=0,$ and $d_\eta\ge 0$ where
$d_\eta:=\gamma_\eta-\int_{(0,1)}y\Pi_\eta(\ud y).$ Thus it suffices
to prove the lemma in the following three cases.

\textbf{Case 1:} Suppose $\sigma_\eta\neq 0.$ Given dependent
Brownian motions $B_\xi$ and $B_\eta$ there exists a Brownian motion
$B'$ independent of $B_\eta$, and constants $a_1$ and $a_2$ such
that $(B_\xi,B_\eta)=(a_1B'+a_2B_\eta~,~B_\eta).$ Using the \LevyIto
decomposition, $X$ can be written as the sum of two independent
processes as follows;
\begin{eqnarray*}X_t=(\xi_t~,~\eta_t)&=&(\xi_t'+B_{\xi,t}~,~\eta_t'+B_{\eta,t})\\
&=_D&(\xi_t'+a_1B_t'~,~\eta_t')+(a_2B_{\eta,t}~,~B_{\eta,t}),\end{eqnarray*}
where $(\xi'~,~\eta')$ is a pure jump \Levy process with drift,
independent of $(B_\xi~,~B_\eta).$ Let $c:=E(\eta_1)$ and define the
\Levy process $Y$ by
$$Y_t:=(\xi_t'+a_1B_t'~,~\eta_t')+\left(a_2(B_{\eta,t}-ct)~,~B_{\eta,t}-ct\right).$$
It is a simple consequence of Girsanov's theorem for Brownian
motion, e.g. Klebaner \cite{Klebaner99}, p.241, that the induced
measures of $B_{\eta,t}$ and $B_{\eta,t}-ct$ on
$\left(D([0,T]\rightarrow\mathbb{R}),\mathscr{B}^{[0,T]}\right)$ are
mutually absolutely continuous. It is trivial to show that this
implies that the induced probability measures of
$(a_2B_{\eta,t}~,~B_{\eta,t})^T$ and
$(a_2(B_{\eta,t}-ct)~,~B_{\eta,t}-ct)^T$ are mutually absolutely
continuous. Using independence, this implies that the induced
probability measures of $X^T$ and $Y^T$ are mutually absolutely
continuous. Note that if we write $Y$ as $Y=(\tau,\nu)$ then
$\nu_t=\eta_t-ct$ so $E(\nu_1)=0$ as required.

\textbf{Case 2:} Suppose $\sigma_\eta=0$ and
$\Pi_\eta((-\infty,0))>0.$ We can assume that $X$ has jumps
contained in $\Lambda,$ a square in $\mathbb{R}^2,$ i.e for all
$t>0$
$$(\Delta\xi_t,\Delta\eta_t)\in\Lambda:=\{(x,y)\in \mathbb{R}^2:-a\le x\le a, -a\le y\le a\}.$$
For any $0<b<a$ define the set $\Gamma\subset\Lambda$ by
$$\Gamma:=\{(x,y)\in\mathbb{R}^2: -a\le x\le a, -a\le y\le -b\}.$$
A \Levy measure is $\sigma$-finite and $\Pi_\eta((-\infty,0))>0$ so
there must exist a $b>0$ small enough such that $\Pi_X(\Gamma)>0.$

By Protter \cite{Protter04}, p.27, we can write
$X=\tilde{X}+\hat{X}$ where
$\tilde{X}_t:=(\tilde{\xi}_t,\tilde{\eta}_t)$ is a \Levy process
with jumps contained in $\Lambda\setminus\Gamma$ and
$\hat{X}_t:=(\hat{\xi}_t,\hat{\eta}_t)$ is a compound Poisson
process independent of $\tilde{X},$ with jumps in $\Gamma$ and
parameter $\lambda:=\Pi_X(\Gamma)<\infty.$ So we can write
$\hat{X}_t=\sum_{i=1}^{N_t}C_i$ where $N$ is a Poisson process with
parameter $\lambda$ and $(C_i)_{i\ge 1}:=(C_i',C_i'')_{i\ge1}$ is an
independent identically distributed sequence of two dimensional
random vectors, independent of $N,$ with $C_i\in\Gamma.$ Let $M$ be
a Poisson process independent of $N,$ $C_i$ and $\tilde{X},$ with
parameter $r\lambda$ for some $r\ge 1.$ Define the \Levy process $Y$
by $Y_t:=\tilde{X}_t+\sum_{i=1}^{M_t}C_i.$ We show the induced
probability measures of $X^T$ and $Y^T$ on
$\left(D([0,T]\rightarrow\mathbb{R}),\mathscr{B}^{[0,T]}\right)$ are
mutually absolutely continuous. Since $\tilde{X}$ is independent of
both compound Poisson processes, this is equivalent to showing the
induced probability measures of $\sum_{i=1}^{N_t}C_i$ and
$\sum_{i=1}^{M_t}C_i$ are mutually absolutely continuous. Let
$A\in\mathscr{B}^{[0,T]}$ and note that
\begin{equation}\label{compound poisson equation}P\left( \left(\sum_{i=1}^{N_t}C_i\right)_{0\le t\le T}\in
A\right)=\sum_{n=0}^\infty P\left(
\left(\sum_{i=1}^{N_t}C_i\right)_{0\le t\le T}\in
A\Big|N_T=n\right)P\left(N_T=n\right).\end{equation} Since $N$ is a
Poisson process, $P(N_t=n)>0$ for all $n\in\mathbb{N}.$ Thus the
left hand side of (\ref{compound poisson equation}) is zero if and
only if $P\left( \left(\sum_{i=1}^{N_t}C_i\right)_{0\le t\le T}\in
A\Big|N_T=n\right)=0$ for all $n\in\mathbb{N}.$

For any Poisson processes, regardless of parameter, Kallenburg
\cite{Kallenberg97}, p.179 shows that once we condition on the event
that $n$ jumps have occurred in time $(0,T]$, then the jump times
are uniformly distributed over $(0,T].$ This implies that
$$P\left( \left(\sum_{i=1}^{N_t}C_i\right)_{0\le t\le T}\in
A\Big|N_T=n\right)=P\left( \left(\sum_{i=1}^{M_t}C_i\right)_{0\le
t\le T}\in A\Big|M_T=n\right).$$ Thus $P\left(
\left(\sum_{i=1}^{N_t}C_i\right)_{0\le t\le T}\in A\right)=0$ if and
only if $P\left( \left(\sum_{i=1}^{M_t}C_i\right)_{0\le t\le T}\in
A\right)=0,$ which proves that the two measures are mutually
absolutely continuous, as required.

Recall that $Y_t=:(\tau_t,\nu_t)=\tilde{X}_t+\sum_{i=1}^{M_t}C_i$
where $\tilde{X}:=(\tilde{\xi},\tilde{\eta})$ and
$C_i:=(C_i',C_i'')\in\Gamma.$ Thus
$\nu_t=\tilde{\eta}_t+\sum_{i=1}^{M_t}C_i''$ which implies that
$tE(\nu_1)=tE(\tilde{\eta}_1)+r\lambda tE(C_i'')$ where
$E(\tilde{\eta}_1)>E(\eta_1)\ge 0.$ Choosing
$r=E(\tilde{\eta}_1)/|\lambda E(C_i'')|$ gives $E(\nu_1)=0$ as
required.

\textbf{Case 3:} Suppose $\sigma_\eta=0,$ $\Pi_\eta((-\infty,0))=0,$
and $d_\eta<0,$ where we allow the possibility that
$d_\eta=-\infty.$ If $\Pi_\eta((0,\infty))=0$ then $\eta_t=d_\eta t$
is deterministic, and this possibility has been excluded. So
$\Pi_\eta((0,\infty))>0$, and we can assume $X$ has jumps contained
in $\Lambda$ where we define the set $\Lambda:=\{(x,y)\in
\mathbb{R}^2:-a\le x\le a, 0< y\le a\}.$ For any $0<b<a$ define the
set $\Gamma^{(b)}\subset\Lambda$ by
$\Gamma^{(b)}:=\{(x,y)\in\mathbb{R}^2: -a\le x\le a, b\le y\le a\}.$

We can write $X=\tilde{X}^{(b)}+\hat{X}^{(b)}$ where
$\tilde{X}^{(b)}:=(\tilde{\xi_t}^{(b)},\tilde{\eta_t}^{(b)})$ is a
\Levy process with jumps contained in $\Lambda\setminus\Gamma^{(b)}$
and $\hat{X}^{(b)}:=(\hat{\xi_t}^{(b)},\hat{\eta_t}^{(b)})$ is a
compound Poisson process independent of $\tilde{X}^{(b)},$ with
jumps in $\Gamma^{(b)}$ and parameter
$\lambda^{(b)}:=\Pi_X(\Gamma^{(b)})<\infty.$

If $d_\eta\in(-\infty,0)$ then
$E\left(\tilde{\eta_t}^{(b)}\right)=d_\eta
t+t\int_{(0,b)}x\Pi_\eta(\ud x).$ Since
$\lim_{b\downarrow0}\int_{(0,b)}x\Pi_\eta(\ud x)=0,$ there exists
$b>0$ such that $E\left(\tilde{\eta_t}^{(b)}\right)<0.$ If
$d_\eta=-\infty$ then $\int_{(0,1)}x\Pi_\eta(\ud x)=\infty.$ Note
that
$E(\eta_1)=E\left(\tilde{\eta_1}^{(b)}\right)+E\left(\hat{\eta_1}^{(b)}\right)\in(0,\infty)$
since jumps are bounded, whilst
$$\lim_{b\downarrow0}E\left(\hat{\eta_t}^{(b)}\right)=\lim_{b\downarrow0}\int_{(b,a)}x\Pi_\eta(\ud
x)=\infty.$$ Hence there again exists $b>0$ such that
$E\left(\tilde{\eta_t}^{(b)}\right)<0.$

From now on we assume $b>0$ is small enough such that
$E\left(\tilde{\eta_t}^{(b)}\right)<0.$ Since a \Levy measure is
$\sigma$-finite and $\Pi_\eta((0,\infty))>0$ we can also assume
$\Pi_X(\Gamma^{(b)})>0.$ Thus we drop the $^{(b)}$ from our
labelling. We can write $\hat{X}_t=\sum_{i=1}^{N_t}C_i$ where $N$ is
a Poisson process with parameter $\lambda$ and $(C_i)_{i\ge
1}:=(C_i',C_i'')_{i\ge1}$ is an independent identically distributed
sequence of two dimensional random vectors, independent of $N,$ with
$C_i\in\Gamma.$ Let $M$ be a Poisson process independent of $N,$
$C_i$ and $\tilde{X},$ with parameter $r\lambda$ for some $r>0.$
Define the \Levy process $Y$ by
$Y_t:=\tilde{X}_t+\sum_{i=1}^{M_t}C_i.$ Then the induced probability
measures of $X^T$ and $Y^T$ are mutually absolutely continuous by
the same proof as used in Case 2. If $Y=:(\tau,\nu)$ then
$\nu_t=\tilde{\eta}_t+\sum_{i=1}^{M_t}C_i''$ with $C_i''\in[b,a].$
Since $E(\tilde{\eta}_1)<0$ for our choice of $0<b<a,$ choosing
$r=|E(\tilde{\eta}_1)|/\lambda E(C_i'')$ gives the result.
\end{proof}

\begin{proof}[Theorem \ref{ruin theorem}]
We first reduce to the case that $X=(\xi,\eta)$ has bounded jumps.
Take a general $(\xi,\eta),$ let $a>0$ and define
$$\Lambda:=\{(x,y)\in \mathbb{R}^2:-a\le x\le a, -a\le
y\le a\}.$$ We can write $X=\tilde{X}+\hat{X}$ where
$\tilde{X}_t:=(\tilde{\xi}_t,\tilde{\eta}_t)$ is a \Levy process
with jumps contained in $\Lambda$ and
$\hat{X}_t:=(\hat{\xi}_t,\hat{\eta}_t)$ is a compound Poisson
process, independent of $\tilde{X},$ with jumps in
$\mathbb{R}^2\setminus\Lambda,$ and parameter
$\lambda:=\Pi_X(\mathbb{R}^2\setminus\Lambda)<\infty.$ Note that
$$\hat{X}_t:=\sum_{0\le s\le t}\Delta X_s1_{\mathbb{R}^2\setminus\Lambda}(\Delta
X_s)$$ and by Poisson properties, $P(\hat{X}_t=0)>0$ for any $t\ge
0.$

 Suppose that $P\left(\int_0^T e^{-\tilde{\xi}_{s-}}\ud
\tilde{\eta}_s<0\right)>0.$ Then $P(Z_T<0)>0,$ because
\begin{eqnarray*}P\left(\int_0^T e^{-\xi_{s-}}\ud
\eta_s<0\right)&\ge& P\left(\int_0^T e^{-\xi_{s-}}\ud
\eta_s<0~\Big|~\hat{X}_T=0\right)P\left(\hat{X}_T=0\right)\\
&=&P\left(\int_0^T e^{-\tilde{\xi}_{s-}}\ud
\tilde{\eta}_s<0~\Big|~\hat{X}_T=0\right)P\left(\hat{X}_T=0\right)\\
&=&P\left(\int_0^T e^{-\tilde{\xi}_{s-}}\ud \tilde{\eta}_s<0\right)P\left(\hat{X}_T=0\right)\\
&>&0.
\end{eqnarray*}
Further, note that $\eta$ is not a subordinator iff we can choose
$a>0$ such that $\tilde{\eta}$ is not a subordinator. If
$\sigma_\eta^2>0$ or $d_\eta<0$ then any $a>0$ suffices. If
$\Pi_\eta((-\infty,0))>0$ then we can choose $a>0$ large enough such
that $\Pi_\eta((-a,0))>0.$ The converse is obvious. Thus the theorem
is proved if we can prove it for the case in which the jumps are
bounded. From now on assume that the jumps of $X=(\xi,\eta)$ are
contained in the set $\Lambda$ defined above. Note that this implies
that $E(\eta_1)$ is finite.

If $\eta$ is pure deterministic drift, then $\eta_t=d_\eta t$ where
$d_\eta<0,$ since $\eta$ is not a subordinator. In this case the
theorem is trivial, since $Z$ is strictly decreasing. Thus, assume
that $\eta$ is not deterministic drift. We first prove the theorem
in the case that $-c:=E(\eta_1)<0.$ Note that
\begin{eqnarray*}P\left(Z_T<0\right)&=&P\left(\int_0^Te^{-\xi_{s-}}\ud(\eta_s+cs)-\int_0^Te^{-\xi_{s-}}\ud(cs)<0\right)\\
&\ge&P\left(\int_0^Te^{-\xi_{s-}}\ud(\eta_s+cs)<0\right)\\
&>&0.
\end{eqnarray*}
The final inequality follows by Lemma \ref{martingale lemma}, which
implies that $\int_0^Te^{-\xi_{s-}}\ud(\eta_s+cs)$ is a martingale,
so $E\left(\int_0^Te^{-\xi_{s-}}\ud(\eta_s+cs)\right)=0.$ Note that
$\int_0^Te^{-\xi_{s-}}\ud(\eta_s+cs)$ is not identically zero due to
our assumption that $\eta$ is not deterministic drift.

Now we assume that $c:=E(\eta_1)\ge0.$ Lemma \ref{existence of Y
lemma} ensures there exists $Y:=(\tau,\nu)$ with bounded jumps,
adapted to $(\Omega,\mathscr{F},\mathbb{F},P),$ such that $X^T$ and
$Y^T$ are mutually absolutely continuous for all $T>0,$ and
$E(\nu_1)=0.$ If we let $W_t:=\int_0^t e^{-\tau_{s-}}\ud \nu_s$ then
Lemma \ref{martingale lemma} ensures that $W^T$ is a mean-zero
martingale. We prove that $W_T$ is not identically zero

Firstly if $\nu$ is deterministic drift then $W$ is either strictly
increasing, or strictly decreasing, hence $W_T$ is not identically
zero. If $\nu$ is not deterministic drift then the quadratic
variation $[\nu,\nu]$ is an increasing process. Hence
$$\Big[\int_0^\bullet e^{-\tau_{s-}}\ud\nu_s,\int_0^\bullet
e^{-\tau_{s-}}\ud\nu_s\Big]_T=\left(\int_0^T e^{-2\tau_{s-}}\ud
[\nu,\nu]_s\right)>0.$$ If $W_T$ is identically zero then $W_t$ must
be identically zero for all $t\le T$, since $W^T$ is a martingale.
Thus $[W,W]_T=0,$ which gives a contradiction.

Since $W$ is not identically zero, and $E(W_T)=0,$ we conclude
$P(W_T<0)>0.$ However, Lemma \ref{measure equivalence lemma} ensures
that the induced probability measures of $Z^T$ and $W^T$ are
mutually absolutely continuous. Hence $P(Z_T<0)>0.$
\end{proof}

Theorem \ref{ruin prob corollary} follows from Theorems \ref{lower
bound theorem} and \ref{sub conditions theorem}. So we now prove
these theorems.

\begin{proof}[Theorem \ref{lower bound theorem}]
Property \ref{p:decrease} is immediate from the definition while
Property \ref{p:monotone} follows from the fact that $V_t$ is
increasing in $z$ for all $t\ge 0.$ Let $W$ be the process such
$e^{-\xi_t}=\epsilon(W)_t$.  Then for any $u\in\mathbb{R},$
\begin{align*}
V_t &= e^{\xi_t}\left(z + \int_{0}^t e^{-\xi_{s-}}
d\eta_s\right)\\
&= e^{\xi_t}\left(z + \int_{0}^t e^{-\xi_{s-}}
d(\eta_s-uW_s) + u\int_{0}^t e^{-\xi_{s-}} dW_s\right)\\
&= e^{\xi_t}\left(z + \int_{0}^t e^{-\xi_{s-}}
d(\eta_s-uW_s) + u(e^{-\xi_{t}}-1)\right)\\
&= u+ e^{\xi_t}\left(z-u+\int_{0}^t e^{-\xi_{s-}}
d(\eta_s-uW_s)\right).
\end{align*}
Now if $\eta-zW$ is a subordinator then $\int_{0}^t e^{-\xi_{s-}}
d(\eta_s-zW_s) \geq 0$ so $\delta(z)=z$.  By Theorem \ref{ruin
theorem} if $\eta-z W$ is not a subordinator then for some $t$ and
some $\epsilon>0$,
\[
P\left(\int_{0}^t e^{-\xi_{s-}} d(\eta_s-zW_s) < -\epsilon\right)
> 0
\]
and so, with $V_0=z+\epsilon$ and $u=z,$
\begin{eqnarray*}
&&P\left(\inf_{t\geq 0} V_t < z \big|V_0=z+\epsilon\right)\\
=&&P\left(\inf_{t\geq 0}\left\{ z+ e^{\xi_t}\left(\epsilon
+\int_{0}^t e^{-\xi_{s-}} d(\eta_s-zW_s) \right)\right\}< z\right)\\
>&&0,\end{eqnarray*} which implies that $\delta(z) \leq
\delta(z+\epsilon) < z$ and establishes Property
\ref{p:subordinator}.

Property \ref{p:subordinator} implies Property \ref{p:fixed} if
$\eta-\delta(z) W$ is a subordinator.  So suppose that
$\eta-\delta(z) W$ is not a subordinator.  Then from the argument
above we know that for some $\epsilon>0$,
$\delta(\delta(z)+\epsilon)<\delta(z)$.  Let $T_u=\inf\{t>0:V_t\leq
u\}$. By definition of $\delta$ we have that
$P(T_{\delta(u)+\epsilon}<\infty)>0$. By the strong Markov property
of $V_t$, if $u<z,$
\begin{eqnarray*}
&&P\left(\inf_{t\geq 0} V_t < \delta(u) \big|V_0=z\right)\\
=&&P\left(\inf_{t\geq 0}
V_{t+T_{\delta(u)+\epsilon}} < \delta(u)\big|V_0=z\right)\\
=&&P\left(\inf_{t\geq 0} V_{t+T_{\delta(u)+\epsilon}} <
\delta(u)\big|T_{\delta(u)+\epsilon}<\infty,V_0=z\right)
P\left(T_{\delta(u)+\epsilon}<\infty\right)\\
\geq&&P\left(\inf_{t\geq 0} V_{t} <
\delta(u)\big|V_0=\delta(u)+\epsilon\right)
P\left(T_{\delta(u)+\epsilon}<\infty\right)\\
>&&0.
\end{eqnarray*}
This contradiction proves Property \ref{p:fixed}.
\end{proof}

\begin{proof}[Proposition \ref{W conditions lemma}]
This proof is similar to a proof in Bertoin et al.
\cite{BertoinLindnerMaller07}. We reference this paper for one of
the tedious calculations. Protter \cite{Protter04}, p.84, proves the
following formula, and shows that it defines a finite valued
semimartingale:
$$\epsilon (W)_t=e^{W_t-\frac{1}{2}[W,W]_t^c}\prod_{0<s\le t}(1+\Delta W_s)e^{-\Delta W_s},$$
where $[W,W]^c$ denotes the path-by-path continuous part of $[W,W].$
Thus
\begin{equation}\label{doleans equation}-\xi_t=\ln\epsilon(W)_t=W_t-\frac{1}{2}[W,W]_t^c+\sum_{0<s\le
t}\left(\ln(1+\Delta W_s)-\Delta W_s\right).\end{equation} So
$\Delta \xi_t=-\ln(1+\Delta W_t)$ whenever $\Delta
W_t\in(-1,\infty),$ and correspondingly, \begin{equation}\label{W
jumps equation}\Delta W_t=e^{-\Delta \xi_t}-1.\end{equation} This
proves the statements concerning the \Levy measures $\Pi_{\xi,W}$
and $\Pi_W.$

It is easy to show that $[W,W]_t^c=\sigma_W^2 t$ whenever $W$ is a
\Levy process. It also follows easily from the definition of the
random measure $N_{W,t}(\cdot,\ud x)$ that
$$\sum_{0<s\le
t}\left(\ln(1+\Delta W_s)-\Delta W_s\right)1_\Lambda(\Delta
W_s)=\int_\Lambda\left(\ln(1+x)-x\right)N_{W,t}(\cdot,\ud x)$$
whenever $0$ is not contained within the closure of $\Lambda.$ Since
Protter \cite{Protter04}, p85, shows that the series $\sum_{0<s\le
t}\left(\ln(1+\Delta W_s)-\Delta W_s\right)$ converges a.s, it
follows that the equality holds for any $\Lambda.$ Hence
(\ref{doleans equation}) becomes
\begin{equation}
\label{doleans equation updated}-\xi_t=W_t-\frac{1}{2}\sigma_W^2 t+
\int_{(-1,\infty)}\left(\ln(1+x)-x\right)N_{W,t}(\cdot,\ud
x).\end{equation} The Brownian motion component of a \Levy process
is independent of the jumps and drift. Thus for equality to hold in
the above equation, we must have $B_W=-B_\xi,$ which proves
(\ref{second covariance matrix equation}).

The proof of (\ref{two dim equation}) closely follows the method of
proving Theorem 2.2 (iv) in \cite{BertoinLindnerMaller07}, and we do
not include it.
\end{proof}
\begin{proof}[Theorem \ref{sub conditions theorem}]
The \Levy process $S^{(u)}:=\eta-uW$ is a subordinator if and only
if the following three conditions hold: $\sigma_{S^{(u)}}^2=0,$
$\Pi_{S^{(u)}}\left((-\infty,0)\right)=0,$ and $d_{S^{(u)}}\ge 0$
where $d_{S^{(u)}}:=E\left(S^{(u)}_1-\int_{(0,\infty)}zN_{S^{(u)},1}(\cdot,\ud
z)\right).$

Note that $\sigma_{S^{(u)}}^2=0$ is equivalent to $B_\eta-uB_W=0,$
which is equivalent to $B_\eta=-uB_\xi$ by Proposition \ref{W
conditions lemma}, which establishes (\ref{covariance matrix
equation}).

We show that $S^{(u)}$ has no negative jumps if and only at least
one of the dot point conditions of the theorem hold. Using (\ref{W
jumps equation}) we see that $\Delta
S^{(u)}_t=\Delta\eta_t-u\left(e^{-\Delta\xi_t}-1\right).$ If $u\ge
0$ then $\Delta S_t^{(u)}<0$ requires $(\Delta \xi_t,\Delta \eta_t)$
be contained within $A_2,$ $A_3,$ or $A_4.$ Every $(\Delta
\xi_t,\Delta \eta_t)\in A_3$ produces a $\Delta S_t^{(u)}<0.$ Recall
that the value $\theta_2$ is the supremum of all the values of $u\ge
0$ at which there can be a negative jump $\Delta S^{(u)}_t$ with
$(\Delta\xi,\Delta\eta)\in A_2.$ Note that at $u=\theta_2$  such a
jump is not possible. The obvious symmetric statement holds for
$\theta_4.$ Hence, if $u\ge 0$ then $S^{(u)}$ can have no negative
jumps if and only if $\Pi_{\xi,\eta}(A_3)=0,$  $\theta_2\le\theta_4$
and $u\in[\theta_2,\theta_4].$

If $u\le 0$ then $\Delta S_t^{(u)}<0$ requires $(\Delta \xi_t,\Delta
\eta_t)$ be contained within $A_1,$ $A_2,$ or $A_3.$ Every $(\Delta
\xi_t,\Delta \eta_t)\in A_2$ produces a $\Delta S_t^{(u)}<0.$ Recall
that the value $\theta_1$ is the supremum of all the values of $u\le
0$ at which there can be a negative jump $\Delta S_t^{(u)}$ with
$(\Delta\xi,\Delta\eta)\in A_1,$ and at $u=\theta_1$  such a jump is
not possible. The obvious symmetric statement holds for $\theta_3.$
Hence, if $u\le 0$ then $S^{(u)}$ can have no negative jumps if and
only if $\Pi_{\xi,\eta}(A_2)=0,$ $\theta_1\le\theta_3$ and
$u\in[\theta_1,\theta_3].$

Finally, if $\Pi_{\xi,\eta}(A_3)=\Pi_{\xi,\eta}(A_2)=0$ then
$\theta_3=\theta_2=0$ and so both of the above are satisfied when
$u\in[\theta_1,\theta_4].$

We now show that when the above two conditions hold, $d_{S^{(u)}}\ge
0$ is equivalent to (\ref{finite drift equation}.) We first use
(\ref{2 dim to 1 dim equation}) to convert (\ref{two dim equation})
into a relationship between the constants $\gamma_\xi$ and
$\gamma_W,$ from the individual characteristic triplets of $\xi$ and
$W.$ It becomes
\begin{equation}\label{one dim equation}\gamma_\xi+\gamma_W=\frac{1}{2}\sigma_\xi^2+\int_\mathbb{R}\left(x1_{(-1,1)}(x)+(e^{-x}-1)1_{(-\ln 2,\infty)}(x)\right)\Pi_\xi(\ud x).\end{equation}
Note that for any Borel set $\Lambda$
\begin{eqnarray*}\int_{\Lambda}zN_{\eta-uW,1}(\cdot,\ud z)&=&\int_{\{x+y\in\Lambda\}}(x+y)N_{-uW,\eta,1}\left(\cdot,\ud
(x,y)\right)\\
&=&\int_{\{y-ux\in\Lambda\}}(y-ux)N_{W,\eta,1}\left(\cdot,\ud
(x,y)\right)\\
&=&\int_{\{y-u(e^{-x}-1)\in\Lambda\}}\left(y-u(e^{-x}-1)\right)N_{\xi,\eta,1}\left(\cdot,\ud
(x,y)\right).
\end{eqnarray*}
The expected value of each of the Brownian motion components of
$\eta$ and $W$ is zero, as is the expected value of the compensated
small jump processes of $\eta$ and $W.$ Thus
\begin{eqnarray*}&&d_{S^{(u)}}\\
=&&E\left(\eta_1-uW_1-\int_{(0,\infty)}zN_{\eta_1-uW_1}(\cdot,\ud z)\right)\\
=&&\gamma_\eta-u\gamma_W+E\bigg(\int_{|y|\ge1}yN_{\eta,1}(\cdot,\ud y)-u\int_{|x|\ge1}xN_{W,1}(\cdot,\ud x)\\
&&~~~~-\int_{(0,\infty)}zN_{\eta_1-uW_1}(\cdot,\ud z)\bigg)\\
=&&\gamma_\eta-u\gamma_W+E\bigg(\int_{|y|\ge1}yN_{\eta,1}(\cdot,\ud y)-u\int_{(-\infty,-\ln2)}\left(e^{-x}-1\right)N_{\xi,1}(\cdot,\ud x)\\
&&~~~~-\int_{\{y-u(e^{-x}-1)>0\}}\left(y-u(e^{-x}-1)\right)N_{\xi,\eta,1}\left(\cdot,\ud (x,y)\right)\bigg)\\
=&&\gamma_\eta+u\gamma_\xi-\frac{1}{2}u\sigma_\xi^2+E\bigg(\int_{\mathbb{R}^2}\Big(y1_{|y|\ge1}-ux1_{|x|<1}-u(e^{-x}-1)\\
&&~~~~-\left(y-u(e^{-x}-1)\right)1_{\{y-u(e^{-x}-1)>0\}}\Big)N_{\xi,\eta,1}(\cdot,\ud(x,y))\bigg)\\
=&&\gamma_\eta+u\gamma_\xi-\frac{1}{2}u\sigma_\xi^2\\
&&~~~~-E\left(\int_{\{y-u(e^{-x}-1)>0\}\cap\{(-1,1)\times(-1,1)\}}(ux+y)N_{\xi,\eta,1}(\cdot,\ud(x,y))\right)\\
=&&\tilde{\gamma}_\eta+u\tilde{\gamma_\xi}-\frac{1}{2}u\sigma_\xi^2\\
&&~~~~-E\left(\int_{\{y-u(e^{-x}-1)>0\}\cap\{x^2+y^2<1\}}(ux+y)N_{\xi,\eta,1}(\cdot,\ud(x,y))\right),
\end{eqnarray*}
where the third equality follows using (\ref{one dim equation}), the
fourth equality follows since $S^{(u)}$ has no negative jumps, so
$N_{\xi,\eta,1}\left(\{y-u(e^{-x}-1)\le0\}\right)=0,$ and the final
equality follows by (\ref{2 dim to 1 dim equation}). Thus we are
done if we can exchange integration and expectation in the above
expression. Now if $f(x,y)$ is a non-negative measurable function
and $\Lambda$ is a Borel set in $\mathbb{R}^2$ then the monotone
convergence theorem implies that
\[
E\left(\int_\Lambda
f(x,y)N_{\xi,\eta,1}(\cdot,\ud(x,y))\right)=\int_\Lambda
f(x,y)\Pi_{\xi,\eta}(\ud(x,y)).
\]
For general $f(x,y),$ if $\int_\Lambda
f^+(x,y)\Pi_{\xi,\eta}(\ud(x,y))$ or $\int_\Lambda
f^-(x,y)\Pi_{\xi,\eta}(\ud(x,y))$ is finite, then the following is
well-defined;
\begin{eqnarray*}
&&E\left(\int_\Lambda
f(x,y)N_{\xi,\eta,1}(\cdot,\ud(x,y))\right)\\
=&&\int_\Lambda f^+(x,y)\Pi_{\xi,\eta}(\ud(x,y))-\int_\Lambda
f^-(x,y)\Pi_{\xi,\eta}(\ud(x,y))\\
=&&\int_\Lambda f(x,y)\Pi_{\xi,\eta}(\ud(x,y)).
\end{eqnarray*}
However, using the fact that $0<e^{-x}-1+x<x^2 $ whenever $|x|< 1,$
we have
\begin{eqnarray*}
&&\int_{\{y-u(e^{-x}-1)>0\}\cap\{x^2+y^2<1\}}(ux+y)^-\Pi_{\xi,\eta}(\ud(x,y))\\
=&&\int_{\{y-u(e^{-x}-1)>0\}\cap\{x^2+y^2<1\}}-(ux+y)1_{\{ux+y\le0\}}\Pi_{\xi,\eta}(\ud(x,y))\\
\le&&\int_{\{y-u(e^{-x}-1)>0\}\cap\{x^2+y^2<1\}}\left(y-u(e^{-x}-1)-(ux+y)\right)1_{\{ux+y\le0\}}\Pi_{\xi,\eta}(\ud(x,y))\\
=&&\int_{\{y-u(e^{-x}-1)>0\}\cap\{x^2+y^2<1\}}-u(e^{-x}-1+x)1_{\{ux+y\le0\}}\Pi_{\xi,\eta}(\ud(x,y))\\
\le&&\int_{\{y-u(e^{-x}-1)>0\}\cap\{x^2+y^2<1\}}|u|x^21_{\{ux+y\le0\}}\Pi_{\xi,\eta}(\ud(x,y))\\
\le&&|u|\int_\mathbb{R}\min\left\{1,x^2\right\}\Pi_\xi(\ud x),\\
\end{eqnarray*}
which is finite since $\Pi_\xi$ is a \Levy measure.
\end{proof}

\begin{proof}[Theorem \ref{ruin prob corollary}]
Clearly $\psi(z)=0$ if and only if $\delta(z)\ge0.$ By Theorem
\ref{lower bound theorem}, this is equivalent to the condition that
there exists $0\le u\le z$ such that $\delta(u)=u.$ Combining this
fact with Theorem \ref{sub conditions theorem} proves Theorem
\ref{ruin prob corollary}.
\end{proof}

\begin{proof}[Theorem \ref{ruin formula theorem}]
Define $$U_t:=e^{\xi_t}(Z_\infty -Z_t)=e^{\xi_t}\int_{t+}^\infty
e^{-\xi_{s-}}\ud \eta_s.$$

Note that since we are integrating over $(t,\infty)$ there are no
predictability problems moving $e^{\xi_t}$ under the integral sign,
as there would have been if we were integrating over $[t,\infty).$
Thus $U_t=\int_{t+}^\infty e^{-(\xi_{s-}-\xi_t)}\ud \eta_s,$ from
which it follows, from \Levy properties, that $U_t$ is independent
of $\mathscr{F}_t$ and that $U_{T_z}$ conditioned on $T_z<\infty$ is
independent of $\mathscr{F}_{T_z}.$

Since $(\xi,\eta)$ is a \Levy process we know that for any $u>0$ and
$t>0$
\begin{equation}\label{equality}(\hat{\xi}_{u-},\hat{\eta}_u)
:=(\xi_{(t+u)-}-\xi_t~,~\eta_{t+u}-\eta_t)=_D(\xi_{u-},\eta_u).\end{equation}
Thus
\begin{eqnarray*}U_t&=&\int_{t+}^\infty e^{-(\xi_{s-}-\xi_t)}\ud \eta_s
=\int_{0+}^\infty e^{-(\xi_{(t+u)-}-\xi_t)}\ud \eta_{t+u}\\
&=&\int_{0+}^{\infty}e^{-(\xi_{(t+u)-}-\xi_t)}\ud
(\eta_{t+u}-\eta_t)
=\int_{0+}^{\infty}e^{-\hat{\xi}_{u-}}\ud\hat{\eta}_u\\
&=&_D \int_{0+}^{\infty}e^{-\xi_{u-}}\ud \eta_u\ \ \
(\textrm{by~}(\ref{equality}))
=Z_\infty \ \ \ \ (\textrm{since}~\Delta\eta_0=0).\\
\end{eqnarray*}
In particular, for any Borel set $A,$
\begin{equation}\label{cond equality} P\left(U_{T_z}\in
A\big|T_z<\infty\right)=P(Z_\infty\in A). \end{equation}
Next note that if $\omega\in\{T_z<\infty\}$ then by definition of
$U,$
\begin{eqnarray*}z+Z_\infty&=&z+Z_{T_z}+e^{-\xi_{T_z}}U_{T_z}\\
&=&e^{-\xi_{T_z}}\left(e^{\xi_{T_z}} (z+Z_{T_z})+U_{T_z}\right)\\
&=&e^{-\xi_{T_z}}(V_{T_z}+U_{T_z}).\\
\end{eqnarray*}
This implies that \begin{equation}\label{subset equality}
P(T_z<\infty,z+Z_\infty<0)=P(T_z<\infty, V_{T_z}+U_{T_z}<0).
\end{equation}
Finally note that $(Z_\infty<-z)\subset(T<\infty)$ since the
convergence from $ Z_t$ to $Z_\infty$ is a.s. Thus
\begin{eqnarray*}P\left(z+Z_\infty<0\right)&=&P\left(T_z<\infty,z+Z_\infty<0\right)\\
&=&P\left(T_z<\infty,V_{T_z}+U_{T_z}<0\right)\ \ \ \ (\textrm{by~}(\ref{subset equality}))\\
&=&E\left(P(T_z<\infty,V_{T_z}+U_{T_z}<0\big|\mathscr{F}_{T_z})\right)\\
&=&\int_{T_z<\infty}P\left(V_{T_z}+U_{T_z}<0\big|\mathscr{F}_{T_z}\right)(\omega)P(\ud\omega).\\
\end{eqnarray*}
But if $T_z(\omega)<\infty$ then
\begin{eqnarray*}P\left(V_{T_z}+U_{T_z}<0\big|\mathscr{F}_{T_z}\right)(\omega)
&=&P\left(V_{T_z}(\omega)+U_{T_z}<0\big|\mathscr{F}_{T_z}\right)(\omega)\\
&=&P\left(U_{T_z}<-V_{T_z}(\omega)\big|T_z<\infty\right)\\
&=&P\left(Z_\infty<-V_{T_z}(\omega)\right)\ \ \ (\textrm{by~}(\ref{cond equality})).\\
\end{eqnarray*}
The second last equality follows since $U_{T_z}$ conditioned on
$T_z<\infty$ is independent of $\mathscr{F}_{T_z}.$ Thus we obtain
the required formula from
\begin{eqnarray*}G(-z)&=&\int_{T_z<\infty}G(-V_{T_z})(\omega)P(\ud\omega)\\
&=&E\left(G(-V_{T_z})1_{T_z<\infty}\right)\\
&=&E\left(G(-V_{T_z})1_{T_z<\infty}\big|T_z<\infty\right)P(T_z<\infty)\\
&&~~~~~~~+E\left(G(-V_{T_z})1_{T_z<\infty}\big|T_z=\infty\right)P(T_z=\infty)\\
&=&E\left(G(-V_{T_z})\big|T_z<\infty\right)P(T_z<\infty).
\end{eqnarray*}
\end{proof}

\textbf{Acknowledgements.} We are grateful to Professor Ross Maller
for several close readings of the paper and constructive comments
which helped us to substantially improve the readability.

\bibliography{bibliography}
\bibliographystyle{plain}

\end{document}